\documentclass[a4paper,reqno,12pt]{amsart}
\usepackage{t1enc}
\usepackage[margin=1.0in]{geometry}
\usepackage{ifthen}
\usepackage{graphicx}
\usepackage{xcolor}
\usepackage{amsmath}
\usepackage{amssymb}
\usepackage{enumerate}
\usepackage{ bbold }

\usepackage{fancyhdr,lastpage}

\newcommand{\R}{\mathbf{R}}

\newcommand{\pr}{\mathbf{P}}

\newcommand{\N}{\mathbb{N}}

\newcommand{\deff}{f_{A,B,t}(s)}

\newcommand{\wolny}{\frac{(xy)^{-\mu-1/2}}{\sqrt{2\pi t}}\exp{\left(-\frac{(x-y)^2}{2t}\right)}}
\newcommand{\zabity}{\frac{(xy)^{-\mu-1/2}}{\sqrt{2\pi t}}\left[\exp{\left(-\frac{(x-y)^2}{2t}\right)}-\exp{\left(-\frac{(x+y-2)^2}{2t}\right)}\right]}
\newcommand{\zabityA}{a^{2\mu+1}\frac{(xy)^{-\mu-1/2}}{\sqrt{2\pi t}}\left[\exp{\left(-\frac{(x-y)^2}{2t}\right)}-\exp{\left(-\frac{(x+y-2a)^2}{2t}\right)}\right]}
\newcommand{\expa}{\exp{\left(-\frac{(x-y)^2}{2t}\right)}}
\newcommand{\expb}{\exp{\left(-\frac{(x+y-2)^2}{2t}\right)}}
\newcommand{\expz}{\exp{\left(-\frac{2(x-1)(y-1)}{t}\right)}}
\newcommand{\expzz}{\exp{\left(\frac{2(x-1)(y-1)}{t}\right)}}

\newcommand{\eb}{\mathbf{E}^{(\mu)}_x}

\theoremstyle{plain}
\newtheorem{theorem}{Theorem}
\newtheorem{lemma}{Lemma}

\newtheorem{proposition}{Proposition}

\theoremstyle{definition}

\theoremstyle{remark}

\DeclareMathOperator\erf{erf}

\newcommand{\formula}[2][nolabel]
{\ifthenelse{\equal{#1}{nolabel}}
 {\begin{align*} #2 \end{align*}}
 {\ifthenelse{\equal{#1}{}}
  {\begin{align} #2 \end{align}}
  {\begin{align} \label{#1} #2 \end{align}}
 }
}

\numberwithin{equation}{section}

\begin{document}

\title[Asymptotic behaviour of the Bessel heat kernels]{Asymptotic behaviour of the Bessel heat kernels}

\author{Kamil Bogus}
\address{Kamil Bogus, Faculty of Pure and Applied Mathematics\\ Wroc{\l}aw University of Science and Technology \\ ul.
Wybrze{\.z}e Wyspia{\'n}\-skiego 27 \\ 50-370 Wroc{\l}aw,
Poland
}
\email{kamil.bogus@pwr.edu.pl}


\thanks{The project was funded by the National Science Centre grant ETIUDA no. 2016/20/T/ST1/00257.}


\begin{abstract}{ We consider Dirichlet heat kernel $p_a^{(\mu)}(t,x,y)$ for the Bessel differential operator  $L^{(\mu)}=\frac{d^2}{dx^2}+\frac{2\mu+1}{2x}$, $\mu\in\R$, in half-line $(a,\infty)$, $a>0$, and provide its asymptotic expansions for $xy/t\rightarrow\infty$.}
\end{abstract}

\maketitle


\section{Introduction}
\label{section:introduction}

Let $\mu\in\R$ and $a>0$. We consider Dirichlet heat kernel $p_a^{(\mu)}(t,x,y)$ in half--line $(a,\infty)$ (with respect to the reference measure $m^{(\mu)}(dy)=m^{(\mu)}(y)dy=y^{2\mu+1}dy$)  for the Bessel differential operator
\formula{
L^{(\mu)}=\frac{d^2}{dx^2}+\frac{2\mu+1}{2x}\frac{d}{dx}.
}
This function is a fundamental solution of the heat equation $2\partial_t=L^{(\mu)}$ with Dirichlet boundary condition at the point $a$. Our main objective was to derive  the precise asymptotic expansion of $p_a^{(\mu)}(t,x,y)$ for $xy/t\rightarrow\infty$. The main result of the paper can be stated as follows.


\begin{theorem}
For $\mu\in\R, \ a>0$ we have
\formula[main]{
p_a^{(\mu)}(t,x,y)=\zabityA\notag\\
\times\left[1+E_{\mu}(t/a^2,x/a,y/a)\right],
}
where  $x,y>a, \ t>0$ and estimates of the error term are given by
\formula[general:bound]{
|E_\mu(t/a^2,x/a,y/a)|\stackrel{\mu}{\lesssim} \frac{t}{a^2}\mathbf{1}_{\{0<t<a^2 t_0(\mu)\}}+\frac{t}{xy}\mathbf{1}_{\{t\geq a^2 t_0(\mu)\}},
}
while $xy/t\rightarrow\infty$, with the constant $t_0(\mu)>0$ defined as follows
\formula[tzero]{
t_0(\mu)=\begin{cases}1 & \phantom{lo}\text{for}\quad |\mu|\leq1/2, \\ \frac{8}{4\mu^2-1} & \phantom{lo}\text{for}\quad |\mu|>1/2.\end{cases}
}
\end{theorem}
 \noindent Here $f\stackrel{\mu}{\lesssim} g$ means that there exists a constant $C=C(\mu)>0$ such that $f\leq Cg$ for all indicated arguments.


It is worth emphasizing,  that in some r\'egime  of the space parameters we are able to give more precise (long--time) estimates of the error  term. In particular,  when  $x,y$ are bounded away from  the boundary,  i.e. level  $a>0$, we can write explicitly the first $n$ (for given $n\in\mathbf{N}$) terms in the expansion of $E_{\mu}(t,x,y)$. Moreover, even when one of the space variables is located near the boundary, we can give  one term more in an expansion of the error term in comparison with \eqref{general:bound}. However, for this purpose, additional assumption  on the quantity $(x-a)(y-a)/t$ is made.

  \begin{theorem}
Let $\mu\in\R, \ a>0, \ n\in\mathbf{N}$ and $t_0(\mu)$ is defined as in \eqref{tzero}. We have 
\formula{
p_a^{(\mu)}(t,x,y)=\zabityA\notag\\
\times\left[1+E_{\mu}(t/a^2,x/a,y/a)\right],
}
and  estimates of the error term for $xy/t\rightarrow\infty$  are given by
\begin{enumerate}[(i)]
\item
$$\left|E_\mu^{(n)}(t/a^2,x/a,y/a)-\sum_{k=0}^{n-1} c_k^{(\mu)}\left(-\frac{t}{xy}\right)^{k}\right|\stackrel{\mu}{\lesssim} \left(\frac{t}{xy}\right)^{n},
$$
whenever $t\geq t_0(\mu)$ and $x,y>2a$. 
 The coefficients $c_k^{(\mu)}$ are expressed explicitly as
\formula[ckmu]{
c_{0}^{(\mu)}=0, \qquad c_k^{(\mu)}=\frac{(4\mu^2-1^2)\ldots(4\mu^2-(2k-1)^2)}{8^k k!}, \quad k\in\mathbf{N},
}
\item 
$$\left|E_\mu^{(2)}(t/a^2,x/a,y/a)-\frac{1-4\mu^2}{8}\frac{t}{xy}\right|\stackrel{\mu}{\lesssim} \left(\frac{t}{xy}\right)^2\/,$$
whenever $ t\geq t_0(\mu),\ a<x<2a<y$ and $\frac{(x-a)(y-a)}{t}\geq 1$.
\end{enumerate}
 \end{theorem}

In order to understand these results better,  notice that the case $a<x,y<2a, \ t\geq t_0(\mu)$  is  excluded from our consideration.  Indeed,  it holds then $xy/t<4a^2/t\leq 4a^2/t_0(\mu)$, which is a contradiction with the condition $xy/t\rightarrow\infty$. Moreover,  since $p_a^{(\mu)}(t,x,y)$ is a symmetric function of space variables $x,y$, i.e.
\formula[p1:symmetric]{
p_a^{(\mu)}(t,x,y)=p_a^{(\mu)}(t,y,x), \quad x,y>a>0, \quad t>0,
}
one can assume, without loss of the generality, that $x<y$. In consequence, it is enough to consider (for $t\geq t_0(\mu)$) two cases: when $a<x<2a<y$ and when $x,y>2a$. Note that one--term, long--time expansion of $E_{\mu}(t,x,y)$ was derived just in one situation, namely  when $a<x<2a<y$ and $0<\frac{(x-a)(y-a)}{t}<1$, while in the other cases we are able to  provide longer expansions of  the error term.

Studying the behaviour of heat kernels  (for second-order differential operators) is of primary importance in many areas of the mathematical analysis, especially in the theory of partial differential equations, harmonic analysis, as well as in the potential theory. However,  even for Laplacian, the problem of providing  precise description of the behaviour of the corresponding Dirichlet heat kernels is very difficult. Exact formulas, in terms of elementary functions, of these heat kernels for Laplace operator are available only in a few special cases (e.g. half--line, $\R^d$ or half--space). As a consequence, providing estimates of the heat kernels is one of the most important approach to describe behaviour of these objects. Some classical results involving short--time estimates of the Dirichlet heat kernel for Laplacian  were derived by S. R. S. Varadhan in \cite{Var1} and \cite{Var2}. It is also worth mentioning that  estimates of heat kernels for Laplace operator in bounded $C^{1,1}$ domains were proved by E. B. Davies in \cite{Davies:1987} (upper bound) and Q. S. Zhang in \cite{Zhang:2002}  (lower bound) -- see also \cite{Zhang2} for the case of domains with bounded complements. These results are only qualitatively sharp, which means that exponential terms in upper and lower bound are different.  Notice that estimates of the  Dirichlet heat kernels for Laplacian in a ball, with the same exponents in upper and lower bounds (so-called sharp estimates), were obtained very recently by J. Ma\l{}ecki and G. Serafin in \cite{MS:2017}.  For more information (including more general second--order  differential operators and domains e.g. manifolds) in this topic we refer the Reader to the monographs \cite{Davies:1990}, \cite{GSC} and references therein.

From the probabilistic point of view,  $p_a^{(\mu)}(t,x,y)$ is a transition probability density of the Bessel process with index $\mu\in\R$ starting from $x>a$ and killed upon leaving a half--line $(a,\infty), \ a>0$. In this context, results in this article are a continuation of the research related to the first hitting times at a given level by the Bessel process. Namely, sharp estimates of the density of these hitting times were obtained recently in \cite{BMR3:2013} and \cite{HM:2013a}. The next natural step 
 was to describe the transition probability density $p_a^{(\mu)}(t,x,y)$. T. Takemura in \cite{Takemura} derived the integral representations involving highly oscillating functions. Then, in papers \cite{BogusMalecki:POTA} and \cite{BogusMalecki:MN}, 
sharp  estimates of $p_a^{(\mu)}(t,x,y)$ in a full range of the variables $x,y>a$ and $t>0$ were obtained. In this  context, providing asymptotic expansion of the considered Bessel heat kernels is a natural improvement of these results. It is worth noting that the analogous expansions for the density of the first hitting time of Bessel processes were derived recently  in \cite{Uchiyama:2015} and \cite{Serafin:2017}.

 To compare the main results  of this paper with the known ones for the heat kernels of the Bessel operator in half--line,  let us invoke the estimates from   \cite{BogusMalecki:POTA,BogusMalecki:MN} related to the case $xy/t\rightarrow\infty$. In particular,  it was shown that for $\mu\in\R$ we have
\formula[POTA:MN]{
p_a^{(\mu)}(t,x,y)\stackrel{\mu}{\approx} 
\left(1\wedge\frac{(x-a)(y-a)}{t}\right)\frac{(xy)^{-\mu-1/2}}{\sqrt{t}}\expa,
}
whenever $x,y>a, \ t>0$ and $xy\geq t$. Here $f\wedge g$ denotes the minimum of $f$ and $g$. In fact, it should be stressed that the estimates  of $p_a^{(\mu)}(t,x,y)$ given in \cite{BogusMalecki:POTA} and \cite{BogusMalecki:MN} are valid for the whole range of the parameters $x,y>a,\ t>0$ and $\mu\in\R$. However, our results are more  precise (but on limited range of parameters)  and the estimate \eqref{POTA:MN} for $xy/t$ large enough is a consequence of those given in Theorem 1. To make it clearer, let us point out the main differences between above mentioned results. At first, notice that the leading term given in \eqref{POTA:MN}, i.e.
 $$
 \frac{(xy)^{-\mu-1/2}}{\sqrt{t}}\expa\/,
 $$
  is of a correct asymptotic order,  but it is not of the exact
asymptotic form.  Secondly, the function 
$$
 1-\exp{\left(-\frac{2(x-a)(y-a)}{t}\right)}
 $$
describes the behaviour of the heat kernel near the boundary (the level $a>0$) more precise  than the expression $1\wedge\frac{(x-a)(y-a)}{t}$ from \eqref{POTA:MN}. This difference is the most apparent when $(x-a)(y-a)/t$ is small due to the expansion of $z\mapsto 1-e^{-z}$ near zero. Finally, our third remark is that, in opposite to \eqref{POTA:MN}, Theorem 1 contains estimates of the error term.  This accuracy strongly depends  (in the case $xy/t\rightarrow\infty$) on the r\'egime of the time parameter (see \eqref{general:bound}). As it was already mentioned, under stronger assumptions on the space--variables $x,y$, even more precise bounds can be obtained -- see Theorem 2.

The paper is organized as follows. In Preliminaries 	basic notation and known results on modified Bessel functions as well as on Bessel processes are introduced. The next two sections are devoted to the proofs of  Theorems 1--2.  In particular, Section 3 contains short--time estimates of $p_a^{(\mu)}(t,x,y)$ given in Theorem 1 (see also Proposition 1), while the long--time ones are provided in series of the Propositions 2--5 in Section 4.  Finally, in Appendix,  some useful lemmas 
  are collected.


\section{Preliminaries}
\label{section:preliminaries}
\subsection{Notation.\\}

\indent The symbol $f\stackrel{\mu}{\lesssim} g$ means that there exists a constant $C=C(\mu)>0$, depending on $\mu$, such that $f\leq Cg$ for all indicated arguments. Moreover, if the conditions $f\stackrel{\mu}{\lesssim}g$ and $g\stackrel{\mu}{\lesssim} f$ are simultaneously satisfied, then we write $f\stackrel{\mu}{\approx}g$. The notation $f=O_{\mu}(g)$ denotes $|f|\stackrel{\mu}{\lesssim} g$.  However, in proofs we omit writing $\mu$ over the signs $\lesssim$ and $\approx$. Similarly, the dependence of the constants on the parameters is not indicated in proofs throughout this article. To simplify the notation, we introduce the function
\formula[def:f]{
 f_{A,B,t}(s)=\frac{1}{\sqrt{s^3}}\frac{1}{\sqrt{t-s}}\exp\left(-\frac{A}{s}\right)\exp\left(-\frac{B}{t-s}\right)
}
where $t>0$ and the coefficients $A, B$ are given by
\formula{
A:=\frac{(x-1)^2}{2}, \quad B:=\frac{(y-1)^2}{2}, \quad x,y>1.
}
Moreover,  the following function (see Theorem 1)
\formula{
g^{(\mu)}(t,x,y)=\zabity,
}
where $x,y>1, \  t>0$ and $\mu\geq0$, will be used in the paper.


\subsection{Modified Bessel function.\\}

\textit{The modified Bessel functions of the first} and \textit{second kind} are denoted by  $I_{\mu}(z)$ and $K_{\mu}(z)$, respectively, and defined as follows
\formula{
I_{\mu}(z)&=\sum_{k=0}^{\infty}\frac{1}{k!\Gamma(\mu+k+1)}\left(\frac{z}{2}\right)^{2k+\mu}, \quad  z>0, \quad \mu>-1, \ \\
K_{\mu}(z)&=\frac{\pi}{2}\frac{I_{-\mu}(z)-I_{\mu}(z)}{\sin(\pi \mu)}, \quad z>0,
}
for $\mu\not\in\mathbf{N}$; otherwise we take a limit when $\mu\to n\in\N$ of the last expression.
Their asymptotic behaviour at infinity can be described as follows (see 9.7.1 and 9.7.2 in \cite{AbramowitzStegun:1972})
\formula[I:infinity]{
 I_{\mu}(z)&=\frac{\exp{(z)}}{\sqrt{2\pi z}}\left[1+\sum_{k=1}^{n-1} c_k^{(\mu)}\left(-\frac{1}{z}\right)^{k}+O_{\mu}\left(\frac{1}{z}\right)^{n}\right], \ n\geq1, \quad z\rightarrow +\infty,\\
 K_{\mu}(z)&=\sqrt{\frac{2}{\pi z}}\exp{(-z)}\left[1+O_{\mu}\left(\frac{1}{z}\right)\right], \quad z\rightarrow+\infty  \label{K:infinity},
 }
where the coefficients $c_k^{(\mu)}$ are defined in \eqref{ckmu}. On the other hand, the asymptotic behaviour of $I_{\mu}(z)$ at zero is given by (see 9.6.7 in \cite{AbramowitzStegun:1972})
 \formula[I:zero]{
 I_{\mu}(z)=\frac{z^{\mu}}{2^{\mu}\Gamma(1+\mu)}\left[1+O_{\mu}(z^2)\right], \quad z\rightarrow 0^+.
 }
 
 Finally, we recall the following two--sided bounds for the ratio of  the modified Bessel function of the first kind, provided by A. Laforgia in \cite{Laforgia:1991}
 \formula[MBF:ineq:upper]{
\left(\frac{x}{y}\right)^{\mu}e^{y-x}\leq\frac{I_\mu(y)}{I_\mu(x)}<\left(\frac{y}{x}\right)^{\mu}e^{y-x}\/,\quad y\geq x>0.\/
}
Here the upper bound holds for $\mu>-1/2$ , while the lower estimate is true for $\mu\geq 1/2$. Moreover, the inequality $I_{\mu}(x)\leq I_{\mu}(y)$ is valid whenever $\mu\geq0$ and $y\geq x>0$ (see formula 8.445 in \cite{bib:Gradshteyn Ryzhlik}).


\subsection{Bessel processes.\\}

We work on the canonical path space $C([0,\infty),\R)$, equipped with the filtration $\{\mathcal{F}_t\}_{t\geq0}$  generated by trajectories up to time $t$. Then, for a given trajectory $R\in C([0,\infty),\R)$, the first hitting time $T_a$ at given level $a>0$ is defined as
$$
T_a=\inf\{t>0: R_t=a\}.
$$
Moreover, we define \textit{the Bessel process} $BES^{(\mu)}(x)$, with index $\mu\in\R$ and starting point $x>0$, as a one--dimensional diffusion on $[0,\infty)$ with infinitesimal generator
$$\frac{1}{2}L^{(\mu)}=\frac{1}{2}\frac{d^2}{dx^2}+\frac{2\mu+1}{2x}\frac{d}{dx}, \quad x>0,$$
where the killing condition  at zero for $\mu<0$ is imposed. The probability law and the corresponding expected value of the Bessel process  $BES^{(\mu)}(x)$ are denoted  by $\mathbf{P}_x^{(\mu)}$ and $\mathbf{E}^{(\mu)}_x$, respectively. Notice that the laws of Bessel processes with different indices are absolutely continuous and their Radon--Nikodym derivative is given by
\formula[ac:formula]{
\left.\frac{d\pr^{(\mu)}_x}{d\pr^{(\nu)}_x}\right|_{\mathcal{F}_t}=\left(\frac{w(t)}{x}\right)^{\mu-\nu}\exp\left(-\frac{\mu^2-\nu^2}{2}\int_{0}^{t}\frac{ds}{w^{2}(s)}\right)\/, 
}
where the equality holds $\pr^{(\nu)}_x$--a.s. on $\{T_0>t\}$.

The density of  $T_a$ for the Bessel process $BES^{(\mu)}(x)$ is denoted by
$$
q_{x,a}^{(\mu)}(s)=\frac{\mathbf{P}_x^{(\mu)}(T_a\in ds)}{ds}, \quad x>a, \ s>0.
$$
For $a=1$ we write simply $q_x^{(\mu)}(s)$, where $x>1$ and $s>0$. Sharp estimates of this function were obtained in \cite{BMR3:2013}, i.e. for $x>1$ and $s>0$ we have
\formula[eq:density:T1:general]{
q_x^{(\mu)}(s)\approx (x-1)\left(1\wedge \frac{1}{x^{2\mu}}\right)\frac{\exp{\left(-\frac{(x-1)^2}{2s}\right)}}{s^{3/2}}\frac{x^{2|\mu|-1}}{(s+x)^{|\mu|-1/2}}, \quad \mu\neq 0,
}
and
\formula[eq:density:T1:general:zero]{
q_x^{(0)}(s)\approx \frac{x-1}{x}\frac{\exp{\left(-\frac{(x-1)^2}{2s}\right)}}{s^{3/2}}\frac{1+\ln x}{1+\ln (s+x)}\frac{(x+s)^{1/2}}{1+\ln (1+s/x)}.
}
However, some more precise estimates of  $q_x^{(\mu)}(s)$ are also available. In  particular, the following asymptotic expansion (see Theorem B in \cite{Uchiyama:BM}) will be used in our article
  \formula[eq:density:T1:3]{
 q_x^{(\mu)}(t)&=\frac{x-1}{\sqrt{2\pi s^3}}x^{-\mu-1/2}\exp{\left(-\frac{(x-1)^2}{2s}\right)}\left(1+O_{\mu}\left(\frac{s}{x}\right)\right), \ x>1, \  t>0.
  }
 This result can be strengthened  by giving more precise estimates of the error term for $0\leq\mu<1/2$ (see Lemma 4 in \cite{BMR3:2013}), namely we have
 \formula[eq:density:T1:2]{
q_x^{(\mu)}(s)=\frac{x-1}{\sqrt{2\pi s^3}}x^{-\mu-1/2}\exp{\left(-\frac{(x-1)^2}{2s}\right)}\left(1+E_{\mu}(s,x)\right),
}
where $0\leq E_{\mu}(s,x)\leq \frac{1-4\mu^2}{8}\frac{s}{x}$ for $x>1, \ s>0$. 
Some improvement is available also in the case $\mu\geq 1/2$, i.e. Lemma 1 in \cite{BogusMalecki:POTA} states that
\formula[eq:density:T1:1]{
q_x^{(\mu)}(s)\leq \frac{x-1}{\sqrt{2\pi s^3}}x^{-\mu-1/2}\exp{\left(-\frac{(x-1)^2}{2s}\right)},
}
where $\mu\geq 1/2$ and $x>1, \ s>0$.

Recall now that $m^{(\mu)}(dy)=m^{(\mu)}(y)dy=y^{2\mu+1}dy$. 
The crucial  fact is that the considered Bessel heat kernel (for $a=0$) is the transition probability density $p^{(\mu)}(t,x,y)$ of the Bessel process $BES^{(\mu)}(x)$
\formula{
p^{(\mu)}(t,x,y)=\frac{\mathbf{E}^{(\mu)}_x [R_t\in dy]}{m^{(\mu)}(dy)}, \quad x,y>0, \quad t>0.
}
 Moreover, this function can be expressed in terms of the modified Bessel function of the first kind in the following way
\formula[eq:pdf]{
p^{(\mu)}(t,x,y) :=p_0^{(\mu)}(t,x,y)= \frac{(xy)^{-\mu}}{t}\exp\left(-\frac{x^2+y^2}{2t}\right)I_{|\mu|}\left(\frac{xy}{t}\right)\/,\quad
x,y> 0, \quad t>0\/.
}
In particular, using \eqref{I:infinity}, one can write
 \formula[pmu:infinity]{
 p^{(\mu)}(t,x,y)=\wolny\left[1+\sum_{k=1}^{n-1} c_k^{(\mu)}\left(-\frac{t}{xy}\right)^{k}+O_{\mu}\left(\frac{t}{xy}\right)^n\right],
 }
 where $x,y>0, \  t>0$ and $xy/t\rightarrow\infty$.
On the other hand,  in the case $a>0$ the considered Bessel heat kernel is the transition probability density of the Bessel process starting from $x>a$ and killed upon leaving a half--line $(a,\infty)$
\formula{
p_a^{(\mu)}(t,x,y)=\frac{\mathbf{E}_x^{(\mu)}[t<T_a; R_t\in dy]}{m^{(\mu)}(dy)}, \quad x,y>a, \quad t>0.
}
We are able to represent the considered heat kernel  $p_a^{(\mu)}(t,x,y)$ in terms of the above mentioned densities $p^{(\mu)}(t,x,y)$ and $q_x^{(\mu)}(s)$, namely it holds
\formula[eq:hunt:formula]{
p_a^{(\mu)}(t,x,y) & = p^{(\mu)}(t,x,y) - r_a^{(\mu)}(t,x,y)\notag\\
&= p^{(\mu)}(t,x,y)- \int_0^t p^{(\mu)}(t-s,a,y)q_{x,a}^{(\mu)}(s)ds\/.
}
In fact, the last formula is a starting point for deriving estimates of the Bessel heat kernels. Notice that putting $\mu=-\nu\geq0$ in \eqref{ac:formula} gives 
\formula[p1:negtive]{
p_a^{(\mu)}(t,x,y)=(xy)^{-2\mu}p_a^{(-\mu)}(t,x,y), \quad x,y>a, \quad t>0,
}
which shows us that it is enough to consider indices $\mu\geq0$. Moreover, we can focus only on the case $a=1$ due to the following scaling property
\formula[p1:scalling]{
p_a^{(\mu)}(t,x,y)=\frac{1}{a}p_1^{(\mu)}(t/a^2,x/a,y/a), \quad x,y>a>0, \ t>0.
}
Finally, let us recall the well--known results in the case $\mu=1/2$ (see formulas 2.13,  2.16 and 2.17 in \cite{BogusMalecki:POTA})
\formula[case:mu12:q]{
q_{x,a}^{(1/2)}(s)&=\frac{x-a}{x}\frac{1}{\sqrt{2\pi s^3}}\exp{\left(-\frac{(x-a)^2}{2s}\right)}
,\\
\label{case:mu12:r1}r_a^{(1/2)}(t,x,y)&=\frac{1}{\sqrt{2\pi t}}\frac{1}{xy}\left(\exp{\left(-\frac{(x+y-2a)^2}{2t}\right)}-\exp{\left(-\frac{(x+y)^2}{2t}\right)}\right), \\
\label{case:mu12:p1}p_a^{(1/2)}(t,x,y)&=\frac{1}{\sqrt{2\pi t}}\frac{1}{xy}\left(\expa-\exp{\left(-\frac{(x+y-2a)^2}{2t}\right)}\right),
}
where $x,y>a>0$ and $t>0$. Observe that the function $p_1^{(1/2)}(t,x,y)$ appears in the expansion of $p_1^{(\mu)}(t,x,y)$ given in the main results of the paper (see Theorems 1--2). Therefore, we will often refer to the formula  \eqref{case:mu12:p1}, especially in proofs of propositions in the next two sections. 


\section{Short--time estimates.}
As it was already mentioned, it is enough to find estimates of $p_a^{(\mu)}(t,x,y)$ only for $a=1$ and $1<x<y$ (see \eqref{p1:scalling} and \eqref{p1:symmetric}). Moreover, due to \eqref{p1:negtive}, one can assume that $\mu\geq 0$. 
We begin with the part of Theorem 1 related to the case $0<t<t_0(\mu)$. 


\begin{proposition}
For $\mu\in\R$ we have
\formula{
p_1^{(\mu)}(t,x,y)=g^{(\mu)}(t,x,y)[1+O_{\mu}(t)],
}
where $x,y>1$ and $0<t<t_0(\mu)$.
\end{proposition}
\begin{proof}
Let $A\subset(1,\infty)$ be a Borel set. We apply the absolute continuity of the laws of Bessel processes \eqref{ac:formula} to write
for $\mu\geq\nu\geq0$
\formula{
\eb[t<T_1,R_t\in A]=\mathbf{E}^{(\nu)}_x\left[t<T_1, R_t\in A;\left(\frac{R_t}{x}\right)^{\mu-\nu}\exp{\left(-\frac{\mu^2-\nu^2}{2}\int_0^t \frac{ds}{R^2(s)}\right)}\right].
}
Notice that  it holds $0\leq \int_0^t \frac{ds}{R^2(s)}\leq\int_0^t ds=t$ on the set $\{t<T_1\}$, which implies 
\formula{
\exp{\left(-\frac{\mu^2-\nu^2}{2}t \right)}\leq\exp{\left(-\frac{\mu^2-\nu^2}{2} \int_0^t \frac{ds}{R^2(s)}\right)}\leq1.
}
Since $\eb[t<T_1,R_t\in A]=\int\limits_A p_1^{(\mu)}(t,x,y)m^{(\mu)}(y)dy$, the set A is given arbitrary and $m^{(\mu)}(y)=y^{2\mu+1}$, we get
\formula{
\exp{\left(-\frac{\mu^2-\nu^2}{2}t \right)}(xy)^{\nu-\mu}p_1^{(\nu)}(t,x,y)\leq p_1^{(\mu)}(t,x,y)\leq (xy)^{\nu-\mu}p_1^{(\nu)}(t,x,y),
}
where $x,y>1$ and $t>0$. Putting $\mu\geq\nu=1/2$ and estimating the above-given exponent from below by $1-\frac{\mu^2-1/4}{2}t$, we obtain
\formula[greater12]{
 \left(1-\frac{\mu^2-1/4}{2}t\right)(xy)^{1/2-\mu} p_1^{(1/2)}(t,x,y)\leq p_1^{(\mu)}(t,x,y)\leq   p_1^{(1/2)}(t,x,y)\/,
}
while for $\mu=1/2\geq\nu\geq0$ it holds
\formula[smaller12]
{
(xy)^{1/2-\nu}p_1^{(1/2)}(t,x,y)\leq p_1^{(\nu)}(t,x,y)\leq \left(1+C_1^{(\nu)}t\right)(xy)^{1/2-\nu}p_1^{(1/2)}(t,x,y).
}
Existence of the constant $C_1^{(\nu)}>0$ is justified by the behaviour of $t\mapsto\exp\left(\frac{\mu^2-\nu^2}{2}t\right)$ for $0<t<1$. To end the proof, it remains to apply the explicit formula for $p_1^{(1/2)}(t,x,y)$,  given in \eqref{case:mu12:p1}, to the inequalities \eqref{greater12} and \eqref{smaller12}.
\end{proof}


\section{Long--time estimates.}

In this section we provide estimates of the Bessel heat kernel $p_a^{(\mu)}(t,x,y)$ for large times $t$. As previously,  we will assume that $a=1, \ 1<x<y$ and $\mu\geq0$. Moreover, as it was  mentioned in Introduction, we need to investigate only two cases: $x,y>2$ and $1<x<2<y$. We start with the first one. In fact, using below given Proposition 2,  the part (i) of Theorem 2 is also proved.


\begin{proposition}
Let $\mu\geq0$  be fixed. Then we have for $x,y>2$ and $t>0$
\formula{
p_1^{(\mu)}(t,x,y)=p^{(\mu)}(t,x,y)\left[1+O_{\mu}\left(\exp{\left(-\frac{xy}{4t}\right)}\right)\right],
}
whenever  $xy/t\rightarrow\infty$.
\end{proposition}
\begin{proof}
Since $p_1^{(\mu)}(t,x,y)\leq p^{(\mu)}(t,x,y)$,  it is enough to focus on the proof of an appropriate lower bound for the considered heat kernel.  We begin with the case $x\geq t$. Notice that, due to the asymptotic behaviour of the modified Bessel function at zero \eqref{I:zero} and at infinity \eqref{I:infinity}, we have for $w>0$ 
\formula[I:upper:general]
{
I_{\alpha}(w)\lesssim w^{\alpha}e^w, \quad \alpha\in\{-1/2, \mu\}.
}
Combining this inequality (for $w=y/(t-s)$ and $\alpha=-1/2$) with the following bound (see \eqref{eq:density:T1:1} and \eqref{eq:density:T1:2}) 
\formula{
q_{x}^{(\mu)}(s)\leq \frac{x-1}{s^{3/2}}\frac{x^{-\mu-1/2}}{\sqrt{2\pi }}\exp{\left(-\frac{(x-1)^2}{2s}\right)}\left(1+\left(\frac{1-4\mu^2}{8}\vee0\right)\frac{s}{x}\right), \quad x>1, \quad s>0,
}
we can write
\formula{
r_1^{(\mu)}(t,x,y)&:=\int_0^t \frac{y^{-\mu}}{t-s}\exp{\left(-\frac{y^2+1}{2(t-s)}\right)}I_{\mu}\left(\frac{y}{t-s}\right)q_{x}^{(\mu)}(s)ds\\
&\lesssim (x-1)\frac{(xy)^{-\mu-1/2}}{2\pi} \int_0^t \deff\left(1+\left(\frac{1-4\mu^2}{8}\vee0\right)\frac{s}{x}\right)ds \\
&\lesssim\frac{(xy)^{-\mu-1/2}}{\sqrt{2\pi t}}\expb\left(1+\left(\frac{1-4\mu^2}{8}\vee0\right)\frac{t}{y}\right)\/.
}
 The last line is a consequence of the exact formula for $\int_0^t f_{A,B,t}(s)ds$ given in Lemma \ref{mu12} and estimates of $\int_0^t sf_{A,B,t}(s)ds$ from Lemma \ref{mu32}. Both of these lemmas can be found in Appendix. Since $xy/t\rightarrow\infty$ we can use \eqref{pmu:infinity} to justify the existence of $C_1>0$ such that
\formula{
r_1^{(\mu)}(t,x,y)\leq C_1 p^{(\mu)}(t,x,y)\expz, 
}
where the inequality $y>t$ was used (since $y>x\geq t$). Moreover, because of $x,y>2$ the last exponent can be bounded from above by $\exp{(-xy/(4t))}$. Hence, the Hunt formula \eqref{eq:hunt:formula} gives us finally
\formula{
p_1^{(\mu)}(t,x,y)\geq p^{(\mu)}(t,x,y)\left(1-C_1\exp{\left(-\frac{xy}{4t}\right)}\right).
}
 Therefore, it remains to deal with the situation when $x<t$. To achieve this, we split integral defining the expression $r_1(t, x, y)$ 
 as follows
$$r_1^{(\mu)}(t,x,y)=\left(\int_0^x+\int_x^t\right) p^{(\mu)}(t-s,1,y)q_x^{(\mu)}(s)ds=\mathcal{J}_1^{(\mu)}(t,x,y)+\mathcal{J}_2^{(\mu)}(t,x,y)$$ 
and then we estimate separately each of these integrals, starting with  $\mathcal{J}_1^{(\mu)}(t,x,y)$. 
Let $0<s<x$. Taking $\alpha=-1/2, \ w=\frac{y}{t-s}$ in \eqref{I:upper:general} and using the following inequality (see \eqref{eq:density:T1:general} and \eqref{eq:density:T1:general:zero}) 
\formula{
q_{x}^{(\mu)}(s)\lesssim \frac{x-1}{s^{3/2}}x^{-\mu-1/2}\exp{\left(-\frac{(x-1)^2}{2s}\right)}, \quad x>2, \quad 0<s<x,
}
we obtain
\formula{
\mathcal{J}_1^{(\mu)}(t,x,y)
&\lesssim (x-1) (xy)^{-\mu-1/2}\int_0^x f_{A,B,t}(s)ds\approx \frac{(xy)^{-\mu-1/2}}{\sqrt{t}}\expb.
}
The last estimate  is a consequence of the exact formula for $\int_0^t f_{a,b}(s)ds$ given in Lemma \ref{mu12} from Appendix. Since $xy/t\rightarrow\infty$ we can use \eqref{pmu:infinity} to justify the following inequalities
\formula[J1:e]{
\mathcal{J}_1^{(\mu)}(t,x,y)\lesssim  p^{(\mu)}(t,x,y)\expz\lesssim  p^{(\mu)}(t,x,y)\exp{\left(-\frac{xy}{4t}\right)},
}
where the second one is a consequence of the assumption $x,y>2$. To deal with the part $\mathcal{J}_2^{(\mu)}(t,x,y)$ we use \eqref{I:upper:general} with $w=\frac{y}{t-s}, \ \alpha=\mu$ and the following upper bound (see again \eqref{eq:density:T1:general} and \eqref{eq:density:T1:general:zero})
\formula{
q_{x}^{(\mu)}(s)\lesssim s^{-\mu-1}\exp{\left(-\frac{(x-1)^2}{2s}\right)}, \quad s>x>2.
}
Namely, we infer that
\formula{
\mathcal{J}_2^{(\mu)}(t,x,y)\lesssim \int_x^t  s^{1/2-\mu} (t-s)^{-\mu-1/2}f_{A,B,t}(s)ds.
}
Moreover, making substitution $w=\frac{1}{s}-\frac{1}{t}$, one can rewrite this inequality as follows
\formula{
\mathcal{J}_2^{(\mu)}(t,x,y)&\lesssim
e^{-\frac{a+b}{t}}t^{-3\mu-1}  \int_{0}^{\frac{1}{x}-\frac{1}{t}} w^{-\mu-1}(1+tw)^{2\mu}e^{-aw}e^{-b/(t^2w)}dw.
}
Therefore, estimating $(1+tw)^{2\mu}\leq (t/x)^{2\mu}$ for $w<\frac{1}{x}-\frac{1}{t}$ and extending the range of integration to $(0,\infty)$, provide us 
\formula{
\mathcal{J}_2^{(\mu)}(t,x,y)&\lesssim 
x^{-2\mu}e^{-\frac{a+b}{t}}t^{-\mu-1} \int_{0}^{\infty} w^{-\mu-1}e^{-aw}e^{-b/(t^2w)}dw\\
&\approx x^{-2\mu}e^{-\frac{a+b}{t}}t^{-\mu-1}\left(\frac{b}{at^2}\right)^{-\mu/2} \frac{\sqrt{t}}{(ab)^{1/4}}e^{-\frac{2\sqrt{ab}}{t}}\\
&\approx p^{(\mu)}(t,x,y) \expz .
}
Here the second  line is a consequence of \eqref{K:integral} and \eqref{K:estimate} (see Appendix), since $\frac{(x-1)(y-1)}{t}\geq\frac{xy}{2t}\rightarrow\infty$ for $x,y>2$ and $xy/t\rightarrow\infty$. On the other hand, the last one follows from asymptotic behaviour of $p^{(\mu)}(t,x,y)$ for $xy/t\rightarrow\infty$ (see \eqref{pmu:infinity}). 
Since we assumed $xy/t\rightarrow\infty$ and $2<x<y$, we get
\formula[J2:e]{
\mathcal{J}_2^{(\mu)}(t,x,y)\lesssim  p^{(\mu)}(t,x,y)  \exp{\left(-\frac{xy}{4t}\right)}.
}
Summarizing, in view of \eqref{J1:e} and \eqref{J2:e}, the usage of the Hunt formula  \eqref{eq:hunt:formula} immediately ends the proof.
\end{proof}


\medskip
\textbf{Proof of the part (i) of Theorem 2.}\medskip\newline
Let us fix $n\in\mathbf{N}$. Due to the asymptotic behaviour of $p^{(\mu)}(t,x,y)$ for $xy/t\rightarrow\infty$  (see \eqref{pmu:infinity}) we can rewrite the result  obtained in Proposition 2 as follows
\formula[pmu:xyawayboundary]{
p_1^{(\mu)}(t,x,y)&=\wolny\left[1+\sum_{k=1}^{n-1} c_k^{(\mu)}\left(-\frac{t}{xy}\right)^{k}+O_{\mu}\left(\frac{t}{xy}\right)^n\right]\notag\\
&\times\left[1+O_{\mu}\left(\exp{\left(-\frac{xy}{4t}\right)}\right)\right],
}
where $x,y>2,\  t\geq 4$ and $xy/t\rightarrow\infty$.
Notice that for $x,y>2$ we have
\formula
{
\frac{\expa}{\expa-\expb}
&=1+O_{\mu}\left(\exp{\left(-\frac{xy}{2t}\right)}\right).
}
Thus we can replace the term $\expa$  in \eqref{pmu:xyawayboundary} by the expression 
$$\left(\expa-\expb\right)\left(1+O_{\mu}\left(\exp{\left(-\frac{xy}{2t}\right)}\right)\right)\/.$$
Moreover, since the exponential growth dominates the power  one, the product of the expressions in square brackets in \eqref{pmu:xyawayboundary} can be  simply replaced by the first of these. Therefore, in this way we have proved the part (i) of Theorem 2.
\begin{flushright}
$\square$
\end{flushright}


The proof of Theorem 1 in the  case $1<x<2<y, \ t\geq t_0(\mu)$ will be provided separately in the situations when $\frac{(x-1)(y-1)}{t}$  is bounded away from zero or bounded away from infinity. To simplify the notation in the proof of the next proposition, we introduce the following functions
\formula{
\mathcal{I}_1^{(\mu)}(t,x,y)&=\int_0^t p^{(\mu)}(t-s,1,y)x^{1/2-\mu}q_x^{(1/2)}(s)ds,\\
\mathcal{I}_2^{(\mu)}(t,x,y)&=\int_0^t sp^{(\mu)}(t-s,1,y)x^{1/2-\mu}q_x^{(1/2)}(s)ds,
}
where $\mu\geq0, \ x,y>1, \ t>0$. Below given estimates  of $\mathcal{I}_1^{(\mu)}(t,x,y)$ and $\mathcal{I}_1^{(\mu)}(t,x,y)$ are crucial to prove the estimate \eqref{general:bound} under the additional  assumptions $\frac{(x-1)(y-1)}{t}\geq1, \ 1<x<2<y, \ t\geq t_0(\mu)$ (see also the part (ii) of Theorem 2).


\begin{proposition}\label{I1I2}
Let $\mu\geq0$. We have 
\formula{
\left\vert\mathcal{I}_1^{(\mu)}(t,x,y)-\frac{(xy)^{-\mu-1/2}}{\sqrt{2\pi t}}\left(1-c_1^{(\mu)}\frac{t}{xy}\right)\expb \right\vert \stackrel{\mu}{\lesssim} \left(\frac{t}{xy}\right)^2 g^{(\mu)}(t,x,y),
}
for $xy/t\rightarrow\infty$, where $1<x<2<y, \ t\geq t_0(\mu)$ and $\frac{(x-1)(y-1)}{t}\geq 1$. It holds also
\formula{
\left\vert \mathcal{I}_2^{(\mu)}(t,x,y)\right\vert\stackrel{\mu}{\lesssim}\left(\frac{t}{xy}\right)^2 g^{(\mu)}(t,x,y),
}
whenever $y>x, \ y>2$ and $xy/t\rightarrow\infty$.
\end{proposition}
\begin{proof}

At first, we deal with $\mathcal{I}_1^{(\mu)}(t,x,y)$. Since $y/(t-s)\geq y/t\rightarrow\infty$ for $1<x<2$ and $xy/t\rightarrow\infty$, asymptotic behaviour of $I_{\mu}(z)$ at infinity \eqref{I:infinity} (for $n=2$) enable us to write 
\formula{
\left\vert
\sqrt{\frac{2\pi y}{t-s}}\exp{\left(-\frac{y}{t-s}\right)}I_{\mu}\left(\frac{y}{t-s}\right)
-\left(1-c_1^{(\mu)}\frac{t-s}{y}\right)
\right\vert
\lesssim \left(\frac{t}{xy}\right)^2
}
or equivalently
\formula{
I_{\mu}\left(\frac{y}{t-s}\right)=\sqrt{\frac{t-s}{2\pi y}}\exp{\left(\frac{y}{t-s}\right)}\left(1-c_1^{(\mu)}\frac{t-s}{y}+O_{\mu}\left(\frac{t}{xy}\right)^2\right),
}
where we have also used $((t-s)/y)^2\lesssim (t/(xy))^2$  for $1<x<2$.
Hence, using exact formula for $q_{x}^{(1/2)}(s)$ (see \eqref{case:mu12:q}) and $f_{a,b,t}(s)$ (see \eqref{def:f}), we obtain
\formula{
\mathcal{I}_{1}^{(\mu)}(t,x,y)&= \mathcal{I}_{1,1}^{(\mu)}+ \mathcal{I}_{1,2}^{(\mu)}\/,
}
where
\formula{
\mathcal{I}_{1,1}^{(\mu)}&=(x-1)\frac{(xy)^{-\mu-1/2}}{2\pi}\left(1-c_1^{(\mu)}\frac{t}{xy}+O_{\mu}\left(\frac{t}{xy}\right)^2\right)\int_0^t  \deff ds,\\
\mathcal{I}_{1,2}^{(\mu)}&=  (x-1)\frac{(xy)^{-\mu-1/2}}{2\pi }\frac{c_1^{(\mu)}}{xy}\int_0^t s \deff ds.
}
Applying Lemma \ref{mu12} from Appendix we get
\formula{
\mathcal{I}_{1,1}^{(\mu)}= \left(1-c_1^{(\mu)}\frac{t}{xy}+O_{\mu}\left(\frac{t}{xy}\right)^2\right)\frac{(xy)^{-\mu-1/2}}{\sqrt{2\pi t}}\expb\/.
}
Moreover, due to \eqref{boundb} one can estimate
\formula[I11e]{
&\left\vert\mathcal{I}_{1,1}^{(\mu)}-\left(1-c_1^{(\mu)}\frac{t}{xy}\right)\frac{(xy)^{-\mu-1/2}}{\sqrt{2\pi t}}\expb\right\vert\notag\\
&\lesssim\left(\frac{t}{xy}\right)^2 \frac{(xy)^{-\mu-1/2}}{\sqrt{t}}\frac{t}{(x-1)(y-1)}\left[\expa-\expb\right]\notag\\
&\lesssim\left(\frac{t}{xy}\right)^2 \frac{(xy)^{-\mu-1/2}}{\sqrt{t}}\left[\expa-\expb\right],
}
where the last inequality is a consequence of $\frac{(x-1)(y-1)}{t}\geq 1$ . On the other hand, we apply Lemma \ref{mu32} from Appendix to calculate $\mathcal{I}_{1,2}^{(\mu)}$ as follows
\formula{
\mathcal{I}_{1,2}^{(\mu)}&=(x-1)\frac{(xy)^{-\mu-1/2}}{2}\frac{c_1^{(\mu)}}{xy} \left[1-\erf\left(\frac{x+y-2}{\sqrt{2t}}\right)\right].
}
Since $1<x<2<y$ and $xy/t\rightarrow\infty$ we get $\frac{x+y-2}{\sqrt{t}}\geq \frac{y-1}{\sqrt{t}}\rightarrow\infty$, which enable us to use \eqref{erf:estimate} and write simply
\formula[I12e]{
|\mathcal{I}_{1,2}^{(\mu)}|
&\approx \frac{x-1}{xy}\frac{(xy)^{-\mu-1/2}}{\sqrt{t}}\frac{t}{x+y-2}\expb\notag\\
&\lesssim \frac{(xy)^{-\mu-1/2}}{\sqrt{t}} \frac{t^2}{y^3} \left[\expa-\expb\right]\notag\\
&\lesssim \frac{(xy)^{-\mu-1/2}}{\sqrt{t}} \left(\frac{t}{xy}\right)^3 \left[\expa-\expb\right].
}
Here the second line follows from $y>2>x$ and the estimate \eqref{boundb}. The last one is valid, because of the assumptions $1<x<2$ and $t\geq t_0(\mu)$. Collecting together \eqref{I11e} and \eqref{I12e} we have
\formula{
\mathcal{I}_{1}^{(\mu)}(t,x,y)
=&\left(1-c_1^{(\mu)}\frac{t}{xy}\right)\frac{(xy)^{-\mu-1/2}}{\sqrt{2\pi t}}\expb\\
+& \frac{(xy)^{-\mu-1/2}}{\sqrt{t}} O_{\mu}\left(\frac{t}{xy}\right)^2 \left[\expa-\expb\right]
}
and it is exactly  what we needed.  To deal with $\mathcal{I}_2^{(\mu)}(t,x,y)$ we apply \eqref{I:infinity} (for $n=1$) together with exact formula for $q_{x}^{(1/2)}(s)$ (see \eqref{case:mu12:q}) to arrive at
\formula{
\mathcal{I}_2^{(\mu)}(t,x,y)\approx (x-1)(xy)^{-\mu-1/2} \int_0^t sf_{a,b}(s)ds.
}
In view of Lemma \ref{mu32} from Appendix (and the estimate given therein)  we can write
\formula{
\mathcal{I}_2^{(\mu)}(t,x,y)
&\approx (x-1)\frac{(xy)^{-\mu-1/2}}{\sqrt{ t}}\frac{t}{x+y-2}\expb\\
&\lesssim \frac{(xy)^{-\mu-1/2}}{\sqrt{ t}} \left(\frac{t}{xy}\right)^2\left[\expa-\expb\right],
}
where the last estimate follows from assumption $1<x<2<y$ and inequality \eqref{boundb} given in Appendix. This ends the proof.
\end{proof}


Now we are in position to deal with the part (ii) of Theorem 2, i.e. when $1<x<2<y, \ t\geq t_0(\mu)$ and $\frac{(x-1)(y-1)}{t}\geq 1$. We follow the notation given in the previous proposition.

\medskip
\textbf{Proof of the part (ii) of Theorem 2.}\medskip\newline
It is enough to use the Hunt formula \eqref{eq:hunt:formula} and estimate each component separately. We start with the subtrahend $r_1^{(\mu)}(t,x,y)$. Namely, using the asymptotic expansion \eqref{eq:density:T1:3} of $q_x^{(\mu)}(s)$, one can write
\formula{
r_1^{(\mu)}(t,x,y)=\mathcal{I}_1^{(\mu)}(t,x,y)+O_{\mu}\left(\frac{\mathcal{I}_2^{(\mu)}(t,x,y)}{x}\right).
}
Hence, Proposition \ref{I1I2} implies that
\formula[r1:est]{
r_1^{(\mu)}(t,x,y)
&=\left(1-\frac{4\mu^2-1}{8}\frac{t}{xy}\right)\frac{(xy)^{-\mu-1/2}}{\sqrt{2\pi t}}\expb\notag\\
&+O_{\mu}\left(\frac{t}{xy}\right)^2\frac{(xy)^{-\mu-1/2}}{\sqrt{t}}g^{(\mu)}(t,x,y).
}
The next step is to provide an appropriate expansion for the expression $p^{(\mu)}(t,x,y)$. Since $xy/t\rightarrow\infty$, applying the formula \eqref{pmu:infinity} for $n=2$ gives us
\formula[pmu:est]{
p^{(\mu)}(t,x,y)
&=\frac{(xy)^{-\mu-1/2}}{\sqrt{2\pi t}}\expa\left[1-\frac{4\mu^2-1}{8}\frac{t}{xy}+O_{\mu}\left(\frac{t}{xy}\right)^2\right]\notag\\
&=\frac{(xy)^{-\mu-1/2}}{\sqrt{2\pi t}}\expa\left[1-\frac{4\mu^2-1}{8}\frac{t}{xy}\right]\notag\\
&+O_{\mu}\left(\frac{t}{xy}\right)^2\frac{(xy)^{-\mu-1/2}}{\sqrt{2\pi t}}g^{(\mu)}(t,x,y),
}
where the last equality follows from \eqref{bounda}. Then, applying \eqref{pmu:est} and \eqref{r1:est} to the Hunt formula (see \eqref{eq:hunt:formula}), we arrive at
\formula{
p_1^{(\mu)}(t,x,y)
&=p^{(\mu)}(t,x,y)-r_1^{(\mu)}(t,x,y)\\
&=\frac{(xy)^{-\mu-1/2}}{\sqrt{2\pi t}}g^{(\mu)}(t,x,y)\left[1-\frac{4\mu^2-1}{8}\frac{t}{xy}+O_{\mu}\left(\frac{t}{xy}\right)^2\right].
}
The proof of the part (ii) of Theorem 2 is complete.
\begin{flushright}
$\square$
\end{flushright}

Finally, we deal with the last case,  namely $1<x<2<y, \ t\geq t_0(\mu)$ and $\frac{(x-1)(y-1)}{t}<1$. In this r\'egime of the variables we provide an upper bound for $p_a^{(\mu)}(t,x,y)$, where $0\leq\mu<1/2$ (see Proposition 4), as well as a lower bound for $\mu\geq 1/2$ (see Proposition 5). Moreover, the proof of Theorem 1 is complete in this case, since  we have for $0\leq\mu<1/2$ (see \eqref{greater12})
\formula{
p_1^{(\mu)}(t,x,y)\geq g^{(\mu)}(t,x,y), \quad x,y>1, \quad t>0,
}
while for $\mu\geq 1/2$ it holds  (see \eqref{smaller12})
\formula{
p_1^{(\mu)}(t,x,y)\leq g^{(\mu)}(t,x,y), \quad x,y>1, \quad t>0.
}


\begin{proposition}
Let $0\leq\mu\leq1/2$. There exists a constant $C_1^{(\mu)}>0$ such that
\formula{
p_1^{(\mu)}(t,x,y)\leq g^{(\mu)}(t,x,y)\left(1+C_1^{(\mu)}\frac{t}{xy}\right),
}
whenever $xy/t\rightarrow\infty$ and $1<x<2, \ y>2, t>t_0(\mu)=1,\ \frac{(x-1)(y-1)}{t}\leq 1$.
\end{proposition}
\begin{proof}
Notice that for $0<s<t$ we have
\formula{
I_{\mu}\left(\frac{y}{t-s}\right)\geq I_{\mu}\left(\frac{y}{t}\right) \quad \text{and} \quad\frac{1}{t-s}\geq\frac{1}{\sqrt{t}\sqrt{t-s}}.
}
Combining these with 
$$
q_x^{(\mu)}(s)\geq \frac{x-1}{\sqrt{2\pi s^3}}x^{-\mu-1/2}\exp{\left(-\frac{(x-1)^2}{2s}\right)}, \quad x>1, \quad s>0, \quad 0<\mu<1/2,
$$
(see \eqref{eq:density:T1:2}) we get
\formula{
r_1^{(\mu)}(t,x,y)&\geq\frac{x-1}{\sqrt{2\pi x t}}(xy)^{-\mu}I_{\mu}\left(\frac{y}{t}\right)\int_0^t f_{a,\frac{y^2+1}{2},t}(s)ds\\
&=\frac{(xy)^{-\mu}}{t}\frac{1}{\sqrt{x}}I_{\mu}\left(\frac{y}{t}\right)\exp\left(-\frac{(x-1+\sqrt{y^2+1})^2}{2t}\right),
}
where the last equality follows from Lemma \ref{mu12} in Appendix. In view of the bound given in \eqref{MBF:ineq:upper} we have
\formula{
I_{\mu}\left(\frac{y}{t}\right)\geq I_{\mu}\left(\frac{xy}{t}\right)\exp{\left(-\frac{(x-1)y}{t}\right)}x^{-\mu},
}
which immediately yields (see also definition of $p^{(\mu)}(t,x,y)$ given in \eqref{eq:pdf})
\formula{
r_1^{(\mu)}(t,x,y)
&\geq\frac{(xy)^{-\mu}}{t}I_{\mu}\left(\frac{xy}{t}\right)x^{-\mu-1/2}\exp\left(-\frac{(x-1)y}{t}\right)\exp\left(-\frac{(x-1+\sqrt{y^2+1})^2}{2t}\right)\\
&=p^{(\mu)}(t,x,y)x^{-\mu-1/2}\exp\left(-\frac{(x-1)(\sqrt{y^2+1}+y-1)}{t}\right).
}
The Hunt formula \eqref{eq:hunt:formula} gives us therefore
\formula{
p_1^{(\mu)}(t,x,y)\leq p^{(\mu)}(t,x,y)\left[1-x^{-\mu-1/2}\exp\left(-\frac{(x-1)(\sqrt{y^2+1}+y-1)}{t}\right)\right].
}
Notice that the expression in square bracket can be decomposed as the sum of 
$$\mathcal{A}_1:=1-\expz$$
 and
 $$\mathcal{A}_2:=\expz-x^{-\mu-1/2}\exp{\left(-\frac{(x-1)(\sqrt{y^2+1}+y-1)}{t}\right)}.$$
 We show that the second one,  divided by $\mathcal{A}_1$, can be bounded from above by $t/(xy)$ (up to some multiplicative positive constant). To achieve this let us denote
\formula{
g(x):=x^{-\mu-1/2}\exp{\left(-\frac{(x-1)(\sqrt{y^2+1}+y-1)}{t}\right)}, \quad x>1,
}
and observe that for $\frac{(x-1)(y-1)}{t}\leq1$ we have (for some $1<x_0<x<2$)
\formula{
\frac{\mathcal{A}_2}{\mathcal{A}_1}&\approx\frac{t}{(x-1)(y-1)}\left(g(1)-g(x)\right)=-\frac{tg'(x_0)}{y-1}.
}
Since $g'(x)=-g(x)\frac{\sqrt{y^2+1}-(y-1)}{t}\left[1+\frac{t}{x}\frac{\mu+1/2}{\sqrt{y^2+1}-(y-1)}\right]$ we have that $g'(x)\approx-g(x)$ for $1<x<2<y, \ t>1$, which implies
\formula{
\frac{\mathcal{A}_2}{\mathcal{A}_1}\approx \frac{tg(x_0)}{y-1}\lesssim \frac{t}{y}.
}
Hence, for some $C_1>0$, we have
\formula{
p_1^{(\mu)}(t,x,y)\leq p^{(\mu)}(t,x,y)\left[1-\expz\right]\left[1+C_1\frac{t}{xy}\right].
}
Finally, notice that due to \eqref{pmu:infinity} we can estimate $p^{(\mu)}(t,x,y)$ for $xy/t\rightarrow\infty$ as follows
$$ p^{(\mu)}(t,x,y)\leq \frac{(xy)^{-\mu-1/2}}{\sqrt{2\pi t}}\expa \left(1+C_2\frac{t}{xy}\right),$$
where $C_2>0$ is a constant. Thus, the proof is complete with the constant $C_1^{(\mu)}:=C_1+C_2>0$.
\end{proof}


\begin{proposition}
Let $\mu>1/2$. There exists a constant $C_2^{(\mu)}>0$ such that
\formula{
p_1^{(\mu)}(t,x,y)\geq g^{(\mu)}(t,x,y)\left(1-C_2^{(\mu)}\frac{t}{xy}\right),
}
whenever $xy/t\rightarrow\infty$ and $1<x<2, \ y>2, t>t_0(\mu)=8/(4\mu^2-1),\ \frac{(x-1)(y-1)}{t}\leq 1$.
\end{proposition}
\begin{proof}
Using \eqref{MBF:ineq:upper} to estimate ratios $I_{\mu}\left(\frac{y}{t-s}\right)/I_{\mu}\left(\frac{y}{t}\right)$ and $I_{\mu}\left(\frac{y}{t}\right)/I_{\mu}\left(\frac{xy}{t}\right)$ from above, we get
\formula{
I_{\mu}\left(\frac{y}{t-s}\right)\leq I_{\mu}\left(\frac{xy}{t}\right)\left(\frac{t}{t-s}\right)^{\mu}\exp{\left(\frac{y}{t-s}\right)}\exp{\left(-\frac{xy}{t}\right)}x^{\mu}.
}
Combining this together with the inequality $q_x^{(\mu)}(s)\leq x^{1/2-\mu}q_x^{(1/2)}(s)$, valid for $s>0$  and $\mu\geq1/2$ (see \eqref{eq:density:T1:1}), we obtain
\formula{
\frac{r_1^{(\mu)}(t,x,y)}{p^{(\mu)}(t,x,y) }\leq x^{\mu-1/2}\exp{\left(\frac{(x-y)^2}{2t}\right)}\frac{x-1}{\sqrt{2\pi}}\sqrt{t}\int_0^t \left(\frac{t}{t-s}\right)^{\mu+1/2}f_{A,B,t}(s)ds,
}
with the function $f_{A,B,t}(s)$ defined in \eqref{def:f}. We split the last integral into three  parts 
\formula{
\mathcal{K}_1&:=\int_0^t \deff ds,\\
\mathcal{K}_2&:=\int_0^{t/2} \left[\left(\frac{t}{t-s}\right)^{\mu+1/2}-1\right]\deff ds,\\
\mathcal{K}_3&:=\int_{t/2}^t \left[\left(\frac{t}{t-s}\right)^{\mu+1/2}-1\right]\deff ds
}
and estimate them separately. To simplify the notation let us denote
\formula{
\mathcal{A}=\mathcal{A}(t,x,y)=x^{\mu-1/2}\exp{\left(\frac{(x-y)^2}{2t}\right)}\frac{x-1}{\sqrt{2\pi}}\sqrt{t}.
}
To deal with $\mathcal{K}_1$ it is sufficient to use Lemma \ref{mu12} from Appendix. Indeed,  we have
\formula[ak1]{
\mathcal{A}\mathcal{K}_1=x^{\mu-1/2}\expz.
}
Now we focus our attention on the expression $\mathcal{K}_2$. Observe that for $s\in(0,t/2)$ we have $t/(t-s))^{\mu+1/2}-1\approx s/(t-s)$. Hence, making substitution $w=\frac{1}{s}-\frac{1}{t}$ in the integral $\mathcal{K}_2$, we can write
\formula{
\mathcal{K}_{2}
&\approx t^{-3/2}e^{-\frac{a+b}{t}}\int_{1/t}^{\infty}w^{-3/2}e^{-aw}e^{-b/(t^2w)}dw
\leq t^{-3/2}e^{-\frac{a+b}{t}}\int_{1/t}^{\infty}w^{-3/2}e^{-b/(t^2w)}dw\\
&\approx t^{-1/2}\cdot\frac{e^{-\frac{a+b}{t}}}{y}\int_0^{b/t}r^{-1/2}e^{-r}dr,
}
where we have substituted $r=b/(t^2w)$. Since $b/t\approx y^2/t\rightarrow\infty$ for $y>2>x$ and $xy/t\rightarrow\infty$, the last integral behaves like a constant. It implies
\formula[ak21]{
\mathcal{A}\mathcal{K}_{2}\lesssim 
\frac{x-1}{y}\expz.
}
To estimate $\mathcal{K}_3$ from above, let us observe that  it holds $(t/(t-s))^{\mu+1/2}-1\approx (t/(t-s))^{\mu+1/2}$, whenever $t/2<s<t$. Hence, putting once again $w=\frac{1}{s}-\frac{1}{t}$ we get
\formula{
\mathcal{K}_{3}
\approx t^{-\mu-1}e^{-\frac{a+b}{t}}\int_0^{1/t}w^{-\mu-1}e^{-aw}e^{-b/(t^2w)}dw.
}
Note that since $a/t\lesssim 1$ we have $e^{-aw}\approx1$. Thus, the substitution $r=b/(t^2w)$ leads to
\formula{
\mathcal{K}_{3}\approx t^{-\mu-1}e^{-\frac{a+b}{t}}\left(\frac{b}{t^2}\right)^{-\mu}\int_{b/t}^{\infty}r^{\mu-1}e^{-r}dr.
}
The last integral is the upper incomplete Gamma function $\Gamma(\mu,b/t)$, which behaves like $(b/t)^{\mu-1}e^{-b/t}$ since $b/t\approx y^2/t\rightarrow\infty$  (see 8.357 in \cite{bib:Gradshteyn Ryzhlik}). Therefore, we have
\formula{
\mathcal{K}_{3}\approx y^{-2}\expb \exp{\left(-\frac{(y-1)^2}{2t}\right)}
}
and consequently
\formula[ak22]{
\mathcal{A}\mathcal{K}_{3}&\approx (x-1)x^{\mu-1/2}\frac{\sqrt{t}}{y^2}\expb \exp{\left(\frac{(x-y)^2}{2t}\right)} \exp{\left(-\frac{(y-1)^2}{2t}\right)}\notag\\
&\lesssim \frac{x-1}{y^2}x^{\mu-1/2}\sqrt{t}\expb
}
due to $\frac{(x-y)^2}{2t}-\frac{(y-1)^2}{2t}=\frac{-(x-1)(2y-x-1)}{2t}<0$. Combining together \eqref{ak1}, \eqref{ak21} and \eqref{ak22} we obtain
\formula{
\frac{r_1^{(\mu)}(t,x,y)}{ p^{(\mu)}(t,x,y)}&\leq \mathcal{A} (\mathcal{K}_1+\mathcal{K}_{2}+\mathcal{K}_{3})\\
&\leq x^{\mu-1/2}\expz\left(1+C_1\frac{x-1}{y}(1+\frac{\sqrt{t}}{y})\right)\\
&\leq x^{\mu-1/2}\expz\left(1+C_2\frac{x-1}{y}\right)
}
for some constants $C_1, C_2>0$. Hence, from the Hunt formula \eqref{eq:hunt:formula}, we obtain
\formula{
\frac{p_1^{(\mu)}(t,x,y)}{p^{(\mu)}(t,x,y)}\geq 1-x^{\mu-1/2}\expz- C_2 \frac{x-1}{y}\expz.
}
Let us decompose the last expression as a sum $\mathcal{B}_1+\mathcal{B}_2-C_2\mathcal{B}_3$, where
\formula{
\mathcal{B}_1&=1-\expz,\\
\mathcal{B}_2&=\expz\left(1-x^{\mu-1/2}\right),\\
\mathcal{B}_3&=\frac{x-1}{y}\expz.
}
Observe that the term $\mathcal{B}_1$ contribute to  the leading part for the expansion of $p_1^{(\mu)}(t,x,y)$. Hence we start  with estimating $\mathcal{B}_2$. One can see that it holds
\formula{
\frac{\mathcal{B}_2}{\mathcal{B}_1}\approx\frac{1-x^{\mu-1/2}}{\expzz-1}\approx -\frac{t}{xy},
}
since $\frac{(x-1)(y-1)}{t}<1$ and $1<x<2<y$. Thus, there exists constant $C_3>0$ such that
\formula[b1]{
\mathcal{B}_1+\mathcal{B}_2\geq \left(1-\expz\right)\left(1-C_3\frac{t}{xy}\right).
}
On the other hand, we can estimate $\mathcal{B}_3$ as follows
\formula
{
\frac{\mathcal{B}_3}{\mathcal{B}_1}\approx\frac{x-1}{y}\frac{1}{\expzz-1}\approx \frac{t}{y^2},
}
with the same justification as for $\mathcal{B}_2.$ This estimate implies that for some $C_4, C_5>0$ we have
\formula[b2]{
\mathcal{B}_2\leq C_4  \frac{t}{y^2}\left(1-\expz\right)\leq C_5\frac{t}{xy}\left(1-\expz\right).
}
Collecting together \eqref{b1} and \eqref{b2} we arrive at
\formula[p1:tw7]{
\frac{p_1^{(\mu)}(t,x,y)}{p^{(\mu)}(t,x,y)}
&\geq \left(1-\expz\right)\left(1-(C_3+C_5)\frac{t}{xy}\right).
}
On the other hand, because of \eqref{pmu:infinity} for $n=1$, there exists $C_6>0$ such that 
\formula[pmu:n1]{
p^{(\mu)}(t,x,y)\geq \frac{(xy)^{-\mu-1/2}}{\sqrt{2\pi t}}\expa \left(1-C_6\frac{t}{xy}\right),
}
where $xy/t\rightarrow\infty$. Putting the last inequality to \eqref{p1:tw7} we get the desired result with the constant   $C_2^{(\mu)}:=C_3+C_5+C_6>0$.
\end{proof}


\section{Appendix}
In this section we collect three lemmas, which play crucial r\^ole in the proofs of the main results of the paper. To make the paper more self--contained we recall the following formula involving the modified Bessel function of the second kind $K_{\mu}(z)$ (see 3.471. 9 in \cite{bib:Gradshteyn Ryzhlik})
\formula[K:integral]{
\int_0^\infty w^{\mu-1}e^{-cw}e^{-d/w}dw=2\left(\frac{d}{c}\right)^{\mu/2}K_{\mu}\left(2\sqrt{cd}\right),
}
which is valid for $c,d>0$ and $\mu\in\R$. Moreover,  since $K_{\mu}(z)\approx z^{-1/2}e^{-z}$ at infinity \eqref{K:infinity}, we obtain for $1\lesssim cd$ that
\formula[K:estimate]{
\int_0^\infty w^{\mu-1}e^{-cw}e^{-d/w}dw\approx c^{-\mu-1/4} d^{\mu/2-1/4} \exp{(-cd)}
}
On the other hand, if we put $\mu=1/2$ in \eqref{K:integral}, then we have
\formula[gr]{
\int_0^\infty w^{-1/2}e^{-cw}e^{-d/w}dw=\sqrt{\frac{\pi}{c}}\exp{\left(-2\sqrt{cd}\right)},
}
 The last equality will be used to prove  Lemma \ref{mu12} given below. Namely,  we derive an explicit formula for the integral $\int_0^t f_{A,B,t}(s)ds$ defined in \eqref{def:f}.


\begin{lemma}\label{mu12}
Let $t>0$. Then we have
\formula{
\int_0^t \frac{1}{\sqrt{t-s}}\frac{1}{\sqrt{s^3}}\exp{\left(-\frac{A}{s}\right)}\exp{\left(-\frac{B}{t-s}\right)}ds=\sqrt{\frac{\pi}{At}}\exp{\left(-\frac{(\sqrt{A}+\sqrt{B})^2}{t}\right)}.
}
\end{lemma}

\begin{proof}
Making the substitution $w=\frac{1}{s}-\frac{1}{t}$, the last integral takes  a form
\formula{
t^{-1/2}\exp{\left(-\frac{c+d}{t}\right)}\int_{0}^{\infty}w^{-1/2}\exp{\left(-cw\right)}\exp{\left( -\frac{d}{t^2}\frac{1}{w}\right)}dw.
}
Using formula  \eqref{gr} we obtain thesis.
\end{proof}


Next Lemma corresponds to $\int_0^t s \deff ds$. We derive exact formula for this  integral in  terms  of the error function.
\begin{lemma}\label{mu32}
Let $c,d>0$ and $t>0$. We have
\formula{
\int_0^t \frac{1}{\sqrt{t-s}}\frac{1}{\sqrt{s}}\exp{\left(-\frac{c}{s}\right)}\exp{\left(-\frac{d}{t-s}\right)}ds=\pi\left[1-\erf\left(\frac{\sqrt{c}+\sqrt{d}}{\sqrt{t}}\right)\right],
}
where $\erf(z)=\frac{2}{\sqrt{\pi}}\int_z^{\infty}\exp{(-u^2)}du, \  z>0,$ denotes the error function. Moreover if $(\sqrt{c}+\sqrt{d})/\sqrt{t}\rightarrow\infty$ then we get
\formula[erf:estimate]{
\int_0^\infty sf_{c,d}(s)ds\approx \frac{\sqrt{t}}{\sqrt{c}+\sqrt{d}}\exp{\left(-\frac{(\sqrt{c}+\sqrt{d})^2}{t}\right)}.
}
\end{lemma}
\begin{proof}
Let $C>0$ be given arbitrary. Integrating the equality from Lemma \ref{mu12} with respect to $c$ on interval $(C,\infty)$, we have
\formula{
\int_C^{\infty}\int_0^t \frac{1}{\sqrt{t-s}}\frac{1}{\sqrt{s}}\exp{\left(-\frac{c}{s}\right)}\exp{\left(-\frac{d}{t-s}\right)}dsdc= \int_C^{\infty}\sqrt{\frac{\pi}{at}}\exp{\left(-\frac{(\sqrt{c}+\sqrt{d})^2}{t}\right)}dc.
}
Since the integrand on the left--hand side (LHS) is positive, we can use the Fubini--Tonelli theorem to change the order of integration as follows
\formula{
\int_C^{\infty}\int_0^t &\frac{1}{\sqrt{t-s}}\frac{1}{\sqrt{s}}\exp{\left(-\frac{c}{s}\right)}\exp{\left(-\frac{d}{t-s}\right)}dsdc\\
&=\int_0^t \left(\int_C^{\infty}\exp{\left(-\frac{c}{s}\right)}da\right)\frac{1}{\sqrt{t-s}}\frac{1}{\sqrt{s}}\exp{\left(-\frac{d}{t-s}\right)}ds\\
&=\int_0^t \frac{1}{\sqrt{t-s}}\frac{1}{\sqrt{s^3}}\exp{\left(-\frac{C}{s}\right)}\exp{\left(-\frac{d}{t-s}\right)}ds.
}
On the other hand, making substitution $w=\frac{\sqrt{c}+\sqrt{d}}{\sqrt{t}}$  in the right--hand side (RHS) we have
\formula{
 \int_C^{\infty}\sqrt{\frac{\pi}{at}}\exp{\left(-\frac{(\sqrt{c}+\sqrt{d})^2}{t}\right)}dc&=
 2\sqrt{\pi}\int_{\frac{\sqrt{C}+\sqrt{d}}{\sqrt{t}}}^{\infty}e^{-w^2}dw
 =\pi\left[1-\erf\left(\frac{\sqrt{C}+\sqrt{d}}{\sqrt{t}}\right)\right].
}
Comparing LHS with RHS give us thesis of the first part of Lemma \ref{mu32}. The second part follows from the asymptotic behaviour of the error function at infinity, namely we have $1-\erf(z)\approx e^{-z^2}/z$ as $z\rightarrow\infty$.
\end{proof}
The last lemma contains some estimates of exponents, which  play important r\^ole especially in the proof of  Proposition 4.


\begin{lemma}\label{expo}
Fix $C>0$ and assume that $x,y>1$ and $t>0$. Then it holds that
\formula[boundb]{
\expb\leq \frac{t\left[\expa-\expb\right]}{2(x-1)(y-1)},
}
while for $\frac{(x-1)(y-1)}{t}\geq C$ we have
\formula[bounda]{
\expa\approx \expa-\expb.
} 
\end{lemma}
\begin{proof}
Since
\formula{
\frac{\expa-\expb}{\expa}=1-\expz
}
and the right--hand side  behaves like positive constant for $\frac{(x-1)(y-1)}{t}\geq C>0$, the proof of \eqref{bounda} is complete. To prove \eqref{boundb} it is enough to  write
\formula{
\expb=\frac{\expa-\expb}{\expzz-1}
}
and simply estimate denominator from below by using the inequality $e^z-1\geq z$, which is true for $z>0$.
\end{proof}


\subsection*{Acknowledgments}
The author is very grateful to Jacek Ma{\l}ecki and Grzegorz Serafin for critical remarks and comments which improve the presentation of the paper.



\begin{thebibliography}{99}

\bibitem{AbramowitzStegun:1972} M. Abramowitz, I. A. Stegun,
\emph{Handbook of Mathematical Functions with Formulas, Graphs, and Mathematical Tables},
9th ed., Dover, New York, (1972).

\bibitem{BogusMalecki:POTA} K. Bogus, J. Ma\l{}ecki,
\emph{Sharp estimates of transition probability density for {Bessel} process in half-line},
Potential Anal. Volume {43}, p. 1--22 (2015).

\bibitem{BogusMalecki:MN} K. Bogus, J. Ma\l{}ecki,
\emph{Heat kernel estimates for the {Bessel} differential operator in half-line},
Math. Nachr. Volume 289, Issue 17-18, p. 2097--2107 (2016).

\bibitem{BMR3:2013} T. Byczkowski, J. Ma{\l}ecki i and M. Ryznar,
\emph{Hitting times of {Bessel} processes}
Potential Anal. Volume {38}, p. 753--786 (2013).

\bibitem{BR:2006} T. Byczkowski, M. Ryznar, 
\emph{Hitting distibution of geometric {Brownian} motion},
Studia Math. Volume {173}, p. 19--38  (2006).

\bibitem{Davies:1987} E. B. Davies,
\emph{The equivalence of certain heat kernel and {Green} function bounds},
J. Funct. Anal. Volume 71, p. 88--103 (1987).

\bibitem{Davies:1990} E. B. Davies,
\emph{Heat kernels and spectral theory (Cambridge Tracts in Mathematics)},
Cambridge Univ. Press 92, Cambridge, pp. 208 (1990).

\bibitem{bib:Gradshteyn Ryzhlik}
I. S. Gradshteyn, I. M. Ryzhik,
\emph{Table of Integrals, Series, and Products},
Academic Press, California, 7th edition, pp. 1171 (2007).

\bibitem{GSC} P. Gyrya, L. Saloff--Coste,
\emph{Neumann and Dirichlet heat kernels in inner uniform domains},
Asterisque, pp. 145 (2011).
 
\bibitem{HM:2013a} Y. Hamana, H. Matsumoto,
\emph{The probability distributions of the first hitting times of {Bessel} processes},
Trans. Amer. Math. Soc. Volume {365}, p. 5237--5257 (2013).

\bibitem{Laforgia:1991}
 A.~Laforgia,
Bounds for modified {Bessel} functions,
  J. Comput. Appl. Math. Volume {34}, p. 263--267 (1991).  

\bibitem{MSZ:2016}J. Ma\l{}ecki, G. Serafin, T. \.Z\'orawik,
\emph{Fourier--Bessel heat kernel estimates},
J. Math. Anal. Appl. Volume 439, Issue 1, p. 91--102 (2016).

\bibitem{MS:2017}J. Ma\l{}ecki, G. Serafin,
\emph{Dirichlet heat kernel for the Laplacian in a ball}
preprint arXiv:1612.01833, pp. 15 (2017). 

\bibitem{RevuzYor:1999} D. Revuz, M. Yor,
\emph{Continuous martingales and Brownian motion},
3rd ed. Springer, New York, pp. 542 (1999).

\bibitem{Serafin:2017} G. Serafin,
\emph{Exit times densities of Bessel process},
Proc. Amer. Math. Soc. Volume 145 (2017), p. 3165--3178 (2017).

\bibitem{Takemura} T. Takemura,
\emph{Elementary solutions of Bessel processes with boundary conditions},
Ann. Reports of Graduate School of Humanities and Sciences, Nara Women's University, Volume 23, p. 265--278 (2008).

\bibitem{Uchiyama:2015} K. Uchiyama,
\emph{Asymptotics of the densities of the first passage time distributions for Bessel diffusions},
Trans. Amer. Math. Soc. Volume 367 Issue 4, p. 2719--2742 (2015).

\bibitem{Uchiyama:BM} K. Uchiyama,
\emph{Density of Space-Time Distribution of Brownian First Hitting of a Disc and a Ball},
Potential Anal. Volume {44}, p. 497--541 (2016).

\bibitem{Var1} S. R. S. Varadhan.
\emph{On the behavior of the fundamental solution of the heat equation with variable coefficients}
Comm. Pure Appl. Math. Volume {20(2)} p. 431--455 (1967).

\bibitem{Var2} S. R. S. Varadhan.
\emph{Diffusion processes in a small time interval.}
Comm. Pure Appl. Math. Volume {20(4)} p. 659--685 (1967).

\bibitem{Zhang:2002} Q. S. Zhang,
\emph{The boundary behavior of heat kernels of {Dirichlet} {Laplacians}},
J. Differential Equations Volume {182}, p. 416--430 (2002).

\bibitem{Zhang2} Q. S. Zhang,
\emph{The global behavior of heat kernels in exterior domains},
J. Funct. Anal., Volume {200}, Issue 1, p. 160--176 (2003).

\end{thebibliography}
\end{document}